\documentclass[11pt]{amsart}
\usepackage{amssymb,latexsym,amsmath}
\input{xypic}
\xyoption{all}

\numberwithin{equation}{section}
\newcommand{\E}{{\mathbb E}}
\newcommand{\Q}{{\mathbb Q}}

\newcommand{\Z}{{\mathbb Z}}
\newcommand{\C}{{\mathbb C}}
\newcommand{\F}{{\mathbb F}}
\newcommand{\A}{{\mathbb A}}

\newcommand{\p}{\mathfrak{p}}
\newcommand{\GL}{{\rm GL}}
\numberwithin{equation}{subsection}
\newtheorem{theorem}[equation]{Theorem}

\newtheorem{lemma}[equation]{Lemma}

\newtheorem{remark}[equation]{Remark}

\newtheorem{example}[equation]{Example}

\begin{document}

\title{Ratios of periods for tensor product motives}

\date{}
\author{\bf Chandrasheel Bhagwat \ \ \& \ \ A. Raghuram}
\address{Indian Institute of Science Education and Research (IISER) \\ Dr.\ Homi Bhabha Road \\ Pashan \\ Pune 411008 \\ India.}
\email{cbhagwat@iiserpune.ac.in, raghuram@iiserpune.ac.in}
\thanks{C.B. is partially supported by DST-INSPIRE Faculty scheme, award number [IFA- 11MA -05].
A.R. is partially supported by the National Science Foundation (NSF), award number DMS-0856113, and an Alexander von Humboldt Research Fellowship. }
\date{\today}
\subjclass[2000]{Primary 11F67; Secondary 11G09}

\begin{abstract}
In this article we prove some period relations for the ratio of Deligne's periods for certain tensor product motives. These period relations give a motivic interpretation for certain algebraicity results for ratios of successive critical values for Rankin-Selberg L-functions for ${\rm GL}_n \times {\rm GL}_{n'}$ proved by G\"unter Harder and the second author.
\end{abstract}

\maketitle

\section{Introduction and motivation}

\subsection{A classical example}
\label{sec:classical}

To motivate the period relations proved in this paper let us recall a classical theorem due to Shimura \cite{shimura}
on the critical values of $L$-functions attached to modular forms. Let $\varphi = \sum a_n q^n$ be a primitive holomorphic cusp form of weight $k$ for $\Gamma_0(N).$ For a Dirichlet character $\chi,$ let $L_f(s, \varphi, \chi)
= \sum_n a_n \chi(n)/n^s.$ There exist $u^{\pm}(\varphi) \in {\mathbb C}^\times$ (the periods of $\varphi$) such that
for any integer $m$ with $1 \leq m \leq k-1$,  we have
$$
L_f(m,\varphi,\chi) \sim (2\pi i)^m \gamma(\chi) u^{\pm}(\varphi),
$$
where $\chi(-1) = \pm(-1)^m$, $\gamma(\chi)$ is the Gau\ss\ sum of $\chi$,
and $\sim$ means equality up to an element of the number field
${\mathbb Q}(\varphi, \chi) := {\mathbb Q}\left(\{a_n\} \cup \{ {\rm values \ of}\ \chi\} \right)$.
Now suppose $1 \leq m \leq k-2$, then
$$
\frac{L_f(m,\varphi,\chi)}{L_f(m+1,\varphi,\chi)} \ \sim \ (2\pi i)^{-1} \frac{u^{\pm}(\varphi)}{u^{\mp}(\varphi)},
$$
assuming the denominator of the left hand side is nonzero. The $2\pi$ on the right hand side can be thought of as coming from the $L$-factors at infinity, and if we define
$\Omega(\varphi) := \tfrac{1}{i} \tfrac{u^+(\varphi)}{u^-(\varphi)},$ then
the ratio of successive critical values  of the completed $L$-function looks like:
\begin{equation}
\label{eqn:ratio-classical}
\frac{L(m,\varphi,\chi)}{L(m+1,\varphi,\chi)} \ \sim \ \Omega(\varphi)^{\chi(-1)(-1)^m}.
\end{equation}
Such an algebraicity result for ratios of successive critical values has been generalized by G\"unter~Harder; 
see \cite{harder} and the references therein to his earlier papers. This was further generalized by Harder and  the second author  \cite{harder-raghuram} which we now briefly recall.

\medskip

\subsection{A generalization}
\label{sec:neoclassical}
Let $\pi$ (resp., $\pi'$) be a cohomological cuspidal automorphic representation of ${\rm GL}_n(\A)$
(resp., ${\rm GL}_{n'}(\A)$), where $\A$ is the ring of ad\`eles of $\Q$.  Implicit in this data is a pure dominant integral highest weight $\lambda$
(resp., $\lambda'$) for the algebraic group $\GL_n/\Q$ (resp.,  $\GL_{n'}/\Q$). Assume that $n$ is even and $n'$ is odd.
Let $\E$ be a number field containing the rationality fields $\Q(\pi)$ and $\Q(\pi').$
To the representation $\pi$ and to any embedding $\iota : E \to \C$, there exist certain {\it relative periods} $\Omega(\pi, \iota) \in \C^\times$ and the collection of these periods, as $\iota$ varies, is well-defined up to
$\E^\times.$ (See  \cite{harder-raghuram}.)
One may say that one has attached $\Omega(\pi) \in (\E \otimes_\Q \C)^\times/\E^\times.$ Suppose that
$m \in \tfrac12 + \Z$ is such that both $m$ and $m+1$ are critical for the Rankin-Selberg $L$-function
$L(s, \pi \times \pi').$ Then under a certain assumption involving only $\lambda$ and $\lambda'$ (called the combinatorial lemma in {\it loc.cit.}), by studying rank one Eisenstein cohomology of $\GL_{n+n'}$, it has been shown that
\begin{equation}
\label{eqn:ratio-automorphic}
\frac{L(m, \pi \times \pi')}{L(m+1, \pi \times \pi')} \ \sim \
\Omega(\pi)^{\epsilon_{\pi'}\epsilon_m}\, c(\pi_\infty, \pi'_\infty),
\end{equation}
where $\sim$ means equality in $(\E \otimes_\Q \C)^\times/\E^\times;$
$\epsilon_{\pi'}$ is a sign depending only on $\pi'$;
$\epsilon_m$ depends only on the parity of the integer $m - \tfrac12$;
$c(\pi_\infty, \pi'_\infty)$ is a nonzero complex number depending only on the representations at infinity (it is expected that this number is rational).
Note a piquant feature: it seems that the representation $\pi$ has a bigger role to play in the right hand side, and that $\pi'$ contributes only a sign in the exponent of $\Omega(\pi).$

\medskip

\subsection{Motivic interpretation}

Every known algebraicity statement on critical values of $L$-functions comes under the umbrella of a celebrated conjecture Deligne \cite{deligne} on critical values of motivic $L$-functions. {\it The purpose of this article is to look at
(\ref{eqn:ratio-classical}) and (\ref{eqn:ratio-automorphic}) from the perspective of Deligne's conjecture.}
Let $M$ be a critical motive over $\Q$ with coefficients in a field $\E$. Then Deligne attaches two periods
$c^{\pm}(M) \in (\E \otimes_\Q \C)^\times/\E^\times$ to $M$ by comparing the Betti and de~Rham realizations of $M$. The finite part of the $\E \otimes_\Q \C$-valued $L$-function $L_f(s,M)$ is defined in terms of the $\ell$-adic realization of $M$. Suppose $s=0$ is critical for the $L$-function then Deligne predicts that $L_f(0,M) \sim c^+(M)$.
Further, if  $s=1$ is also critical then
$L_f(1,M) \sim (2\pi i)^{d^-(M)} c^-(M)$ for a certain $d^-(M) \in \Z$.
Hence, the ratio $L_f(0,M)/L_f(1,M) \sim (2\pi i)^{-d^-(M)} c^+(M)/c^-(M).$
As in the case of modular forms, the power of $2\pi$ can be interpreted as a relevant ratio of $L$-factors from infinity (which depend only on the Hodge types of $M$). Hence if we wish to consider a ratio of successive critical values then we need to consider the ratio of periods $c^+(M)/c^-(M)$. To see (\ref{eqn:ratio-automorphic}), assume that $M$ is pure and of even rank, and consider another pure motive $M'$ whose rank is odd. Assume that $M \otimes M'$ is critical. Further, assume that all the nonzero Hodge numbers of $M$ and $M'$ are $1.$
One of the main results of this paper (Theorem~\ref{thm:main-q}) states that in
$(\E \otimes_\Q \C)^\times/\E^\times$ we have
\begin{equation}
\label{eqn:ratio-motivic}
\frac{c^{+}(M\otimes M')}{c^{-}(M\otimes M')}  \ = \ \left( \frac{c^{+}(M)}{c^{-}(M)}\right)^{\epsilon(M')},
\end{equation}
where $\epsilon(M')$ is the sign by which complex conjugation acts on the  middle Hodge type of $M'$.

\medskip

The proof of these period relations is based on the formalism of Yoshida \cite{Yo} on periods of tensor product motives which need not only $c^\pm(M)$ but also other invariants attached to $M$. In \S\ref{sec:yoshida} we recall the relevant parts of his paper that we need. \S\ref{sec:main-q} and \S\ref{sec:motives-tot-real} are the two main sections of this article.

\medskip

In \S\ref{sec:parity} we prove a couple of variations of (\ref{eqn:ratio-motivic}) when the ranks of both the motives have the same parities. We start with the easy case when both $M$ and $M'$ have even rank; in this situation we have the following period relation in $(\E \otimes_\Q \C)^\times/\E^\times$:
\begin{equation}
\label{eqn:ratio-motivic-both-even}
c^{+}(M\otimes M') \ = \ c^{-}(M\otimes M').
\end{equation}
The relations (\ref{eqn:ratio-motivic}) and (\ref{eqn:ratio-motivic-both-even}) are combined together
in Theorem~\ref{thm:combined-even-odd} where we present a period relation when $M$ has even rank, and $M'$ is a direct sum of critical motives. Next, we consider the somewhat more difficult case when both $M$ and $M'$ have odd rank; in this situation we have the following equality in $(\E \otimes_\Q \C)^\times/\E^\times$:
\begin{equation}
\label{eqn:ratio-motivic-both-odd}
\frac{c^{+}(M\otimes M')}{c^{-}(M\otimes M')}  \ = \ \left( \frac{c^{+}(M)}{c^{-}(M)}\right)^{\epsilon(M')} \left( \frac{c^{+}(M')}{c^{-}(M')}\right)^{\epsilon(M)}.
\end{equation}
See Theorem~\ref{thm:main-odd}.

There has been a long tradition, since the foundational paper by Deligne \cite{deligne}, on period relations for motivic periods and what such relations say about the special values of automorphic $L$-functions; see, for example, 
Blasius~\cite{blasius-orloff}, Harris~\cite{harris97}, Panchishkin~\cite{Pa}, and Yoshida~\cite{Yo}. The reader should view this article from the perspective of such a tradition.

\bigskip
{\small
\noindent{\it Acknowledgements:} It is a pleasure to thank Dipendra Prasad who made several suggestions in his attempts to understand the motivic period relations. The comments in \S\ref{sec:hodge-one} and the formulation of
Theorem~\ref{thm:combined-even-odd} were suggested by him. We thank the referee for his/her meticulous comments which helped us enormously in revising the article.}

\section{Special polynomials, Motivic periods, and results of Yoshida}
\label{sec:yoshida}

In this section we begin by briefly reviewing the notion of a critical motive. Then we review some results of Yoshida \cite{Yo} that will be useful for our proofs.

\subsection{Critical motives}

Let $M$ be a pure motive defined over $\Q$ with coefficients in a number field $\mathbb{E}$. Every pure motive over 
$\Q$ conjecturally arises, up to a Tate twist, as a submotive of the cohomology motive $H^{\ast}(X)$ of an algebraic variety $X$ over $\Q$. In this paper we consider the motives in the sense of their Betti, de~Rham and $\ell$-adic realizations as in Deligne \cite{deligne}.

Let $H_{B}(M)$ be the
{\it Betti realization} of $M$. It is a finite-dimensional vector space over $\E$. The rank $d(M)$ of $M$ is defined to be $\text{dim}_{\E}(H_{B}(M))$. Write
\[ H_{B}(M) \ = \ H^{+}_{B}(M) \oplus H^{-}_{B}(M), \]
where $H^{\pm}_{B}(M)$ are the $\pm 1$-eigenspaces for the action of complex conjugation $\rho$ on
$H_{B}(M)$. Let $d^{\pm}(M)$ be the $\E$-dimension of $H^{\pm}_{B}(M)$. The Betti realization has a
{\it Hodge decomposition}:
\begin{equation}
\label{eqn:hodge-decomp}
H_{B}(M) \otimes_{\Q} \C \ = \ \bigoplus_{p,q \in \Z} H^{p,q}(M),
\end{equation}
where $H^{p,q}(M)$ is a free $\E \otimes \C$-module of rank $h^{p,q}_M$. The numbers $h^{p,q}_M$ are called the {\it Hodge numbers} of
$M.$  Purity of $M$ means that there is an integer $w$ such that
 $H^{p,q}(M) = \left\{0\right\}$ if $p+q \neq w.$ Henceforth, we assume that all our motives are pure.
The number $w$ is called the weight of $M$. We also have $\rho(H^{p,q}(M)) = H^{q,p}(M);$ and hence $\rho$ acts on the (possibly zero) middle Hodge type $H^{w/2,w/2}(M).$

\smallskip

Let  $H_{DR}(M)$ be the {\it de~Rham realization} of $M$; it is a $d(M)$-dimensional vector space over $\mathbb{E}$. There is a comparison isomorphism of $\E \otimes_\Q \C$-modules:
\[
I : H_{B}(M) \otimes_{\Q} \C  \ \longrightarrow \ H_{DR}(M) \otimes_{\Q} \C.
\]
The de~Rham realization has a {\it Hodge filtration} $F^p(M)$ which is a decreasing filtration of $\E$-subspaces of $H_{DR}(M)$ such that
\[I \left(\bigoplus_{p' \geq p} H^{p',q}(M)\right) \ = \  F^{p}(M) \otimes_{\Q} \C.\]
Write the Hodge filtration as
\begin{equation}
\label{Hodge}
H_{DR}(M) = F^{p_1}(M) \supset F^{p_2}(M) \supset \cdots \supset F^{p_m}(M) \supset F^{p_{m}+1}(M) = \left\{0\right\};
\end{equation}
all the inclusions are proper and there are no other filtration-pieces between two successive members. The numbers $p_\mu$ are such that, $h^{p_{\mu}, w-p_\mu}_M \neq 0$. We assume that the numbers $p_\mu$ are maximal among all the choices.
Let $s_\mu = h^{p_\mu, w-p_\mu}_M$ for $1 \leq \mu \leq m.$ Purity plus the action of complex conjugation on Hodge types says that the numbers $p_j$  and $s_\mu$ satisfy
$p_{j} + p_{m+1-j} = w, \forall\  1 \leq j \leq m,$ and
$s_\mu = s_{m+1-\mu}, \forall\  1 \leq \mu \leq m.$

\smallskip

We say that the motive {\it $M$ is critical} if there exist $\mathfrak{p}^+, \mathfrak{p}^- \in \Z$ such that
\[ \sum \limits_{i=1}^{\mathfrak{p}^+} s_i= d^+(M), \quad \sum \limits_{i=1}^{\mathfrak{p}^-} s_i= d^-(M). \]
In this case one says that $F^{\pm}(M)$ exists and equals $F^{\mathfrak{p}^{\pm}}(M)$.
It is easy to see that $M$ is critical if and only if complex conjugation acts by a scalar on the middle Hodge type (provided the middle Hodge type exists); in this situation we denote this scalar by $\epsilon(M)$. 

\smallskip

\subsection{Period invariants}
\label{sec:yoshida-polynomials}
The period matrix of $M$ is defined in terms of $\E$-bases for the spaces $H^{\pm}_{B}(M)$ and $H_{DR}(M)$.
Let $\left\{v_{1},v_{2},\ldots,v_{d^{+}(M)}\right\}$ be an $\E$-basis of $H^{+}_{B}(M)$, and similarly,
$\left\{v_{d^{+}(M) + 1},v_{d^{+}(M) + 2},\ldots,v_{d(M)}\right\}$ be an $\E$-basis of $H^{-}_{B}(M)$. Let $\left\{w_{1},w_{2},\ldots,w_{d(M)}\right\}$ be a basis of $H_{DR}(M)$ over $\E$ such that $\left\{w_{s_{1}+s_{2}+ \ldots +s_{\mu -1}+1},\ldots,w_{d(M)}\right\}$ is a basis of $F^{p_\mu}(M)$ for $1 \leq \mu \leq m$.
The period matrix $X$ of $M$ is the matrix which represents the comparison isomorphism between the two realizations of $M$ with respect to the bases chosen above. 

\smallskip

Let $\mathbb{F}$ be a number field. Suppose $d$ is a positive integer. Fix a partition $s_1 + s_2 + \ldots + s_m = d$. Let $P_m$ be the corresponding lower parabolic subgroup of $\GL(d,F)$. Given an $m$-tuple of integers $(a_i)_{1 \leq i \leq m}$, define an algebraic character $\lambda_1$ of $P_m$ by
\[
\lambda_1 \left( {\begin{pmatrix} p_{11}& 0 & \ldots & 0 \\ \ast & p_{22} & \ldots & 0\\\ast & \ast & \ddots & \ldots \\ \ast & \ast & \ast & p_{mm} \end{pmatrix}} \right) = \prod \limits_{1 \leq i \leq m}(\text{det}~p_{ii})^{a_i};  \quad p_{ii} \in \GL(s_i).
\]
Let $d = d^+ + d^-$. Given $k^+, k^- \in \Z$, define a character $\lambda_2$ of $\GL(d^+) \times  \GL(d^-)$ by
\[
\lambda_2 \left( \begin{pmatrix} a & 0 \\ 0 & b \end{pmatrix} \right) = ~(\text{det}~a)^{k^+} (\text{det}~b)^{k^-}, \quad a \in \GL(d^+),~ b \in \GL(d^-).
\]

Let $f(x)$ be a polynomial with rational coefficients which satisfies the following equivariance condition with respect to the left action of $P_m$ and the right action of $\GL(d^+) \times \GL(d^-)$ on the matrix ring $M_{d}(F)$:
\begin{equation}
\label{inv}
f(p x \gamma) = \lambda_1(p) f(x) \lambda_2({\gamma}),  \quad \forall~ p \in P_m, \ \
\forall~ \gamma \in \GL(d^+) \times \GL(d^-).
\end{equation}
A polynomial satisfying (\ref{inv})
is said to have admissibility type $\left\{(a_1, a_2, \ldots, a_m), (k^+, k^-)\right\}$.
Yoshida \cite[Theorem 1]{Yo} proves that  the space of polynomials of a given admissibility type is atmost one.

 \begin{lemma}\label{multi}
 If the polynomial $f(x)$ has admissibility type $\left\{(a_1, a_2, \ldots, a_m), (k_1^+, k_1^-)\right\}$,
 and $g(x)$  has admissibility type  $\left\{(b_1, b_2, \ldots, b_m), (k_2^+, k_2^-)\right\}$, then the polynomial $h(x) = f(x)g(x)$ has admissible type is given by
 \[
 \left\{(a_1 + b_1, a_2 + b_2, \ldots, a_m + b_m), (k_1^+ + k_2^+, k_1^- + k_2^-)\right\}.
 \]
 \end{lemma}

\begin{proof} Follows from (\ref{inv}).
\end{proof}

\begin{example} Let $f(x) = {\rm det}(x)$ for a matrix $x \in M_{d}(F)$. Then $f(x)$ is of admissibility type
$\left\{(1,1,1,\ldots,1),(1,1)\right\}$.
\end{example}

 Let $f^{\pm}(x)$ be the upper left  (resp., upper right) $d^{\pm} \times d^{\pm}$ determinant of $x$.
 Then it can be seen that the admissibility types of $f^+(x)$ and $f^-(x)$ are respectively given by
 \[
 \{(\underbrace{1,1,1,\ldots,1}_{\mathfrak{p}^{+}},0,0, \ldots, 0),(1,0)\},
 \]
 \[
 \{(\underbrace{1,1,1,\ldots,1}_{\mathfrak{p}^{-}},0,0, \ldots, 0),(0,1)\}.
 \]

Yoshida interprets the period invariants of the period matrix $X$ via some special polynomials as $\delta(M) = f(X)$ and $c^{\pm}(M) = f^{\pm}(X)$.
The values in $(\E\otimes \C)^\times$ of these polynomials on a period matrix defines an element of 
$(\E\otimes \C)^\times/\E^\times$ which is independent of the de~Rham and Betti bases chosen to define the period matrix.

\smallskip

\subsection{Tensor product motives}

The category of motives over $\Q$ with coefficients in $\E$ admits a tensor product. The realizations for the tensor product are naturally identified with the tensor products of the realizations.  Yoshida \cite[Proposition 12]{Yo} describes the admissibility types of the polynomials which correspond to the periods $c^{\pm}( M  \otimes M')$ of the tensor product motive $M \otimes M'$.

Let $X$ and $Y$ be the period matrices of the motives $M$ and $M',$ respectively. Let $R = \E \otimes_\Q \C$. Suppose $d(M)$ and $d(M')$ are the ranks of $M$ and $M'$ respectively. The numbers $d^{\pm}(M')$ are the dimensions of the 
$\pm 1$-eigenspaces for $H_{B}(M')$.  Write $X$ and $Y$ in the following way:
\[
X = \begin{pmatrix} X_1^+ & X_1^- \\  X_2^+ & X_2^- \\ \vdots & \vdots \\ X_{d(M)}^{+} & X_{d(M)}^{-} \end{pmatrix}, \ \ \
~Y = \begin{pmatrix} Y_1^+ & Y_1^- \\  Y_2^+ & Y_2^- \\ \vdots & \vdots \\ Y_{d(M')}^{+} & Y_{d(M')}^{-} \end{pmatrix},
\]
where $X_{i}^\pm \in R^{d^{\pm}(M)} ~\text{and}~ Y_l^\pm \in R^{d^{\pm}(M')}$.

Given $1 \leq i \leq d(M)$, let $1 \leq \mu \leq m$ be such that
\[ s_1 + s_2 + \ldots s_{\mu-1} < i \leq s_1 + s_2 + \ldots + s_\mu.\]
The Hodge level $w(X_i)^{\pm}$ is defined to be the integer $p_\mu$ (\textit{cf.} (\ref {Hodge})). The Hodge level $w(Y_l)^{\pm}$
 is defined analogously. Suppose the motive $M \otimes M'$ is critical.
 Consider the Hodge filtrations of the motives $M$, $M'$ and $M \otimes M'$.
 \[ H_{dR}(M) = F^{i_1}(M) \supsetneq F^{i_2}(M) \supsetneq \ldots \supsetneq F^{i_{m_1}}(M) \supsetneq (0),\]
\[ H_{dR}(M') = F^{j_1}(M') \supsetneq F^{j_2}(M') \supsetneq \ldots \supsetneq F^{j_{m_2}}(M') \supsetneq (0),\]
\[ H_{dR}(M \otimes M') = F^{k_1}(M \otimes M') \supsetneq F^{k_2}(M \otimes M') \supsetneq \ldots \supsetneq F^{k_m}(M \otimes M') \supsetneq (0).\]

Let $u_i$ denote the Hodge numbers of $M \otimes M'$.  Hence there exist integers $q^{+}, q^{-}$ such that
\begin{equation} \label{hodgetensor} u_1 + u_2 + \ldots + u_{q^{\pm}} = d^{\pm}(M \otimes M'). \end{equation} The integers $a^{\pm}_{\mu}$ are defined by
\begin{equation} \label{amu} a^{\pm}_{\mu} \quad = \quad \left|\left\{ l : 1 \leq l \leq d(M'),~ p_\mu + w(Y_l) < k_{q^{\pm}} \right\}\right|.
\end{equation}

From the definition of the period, it follows that $c^{+}(M \otimes M')$ is the determinant of the square matrix $Z^+$ of
 size $d^{+}(M) d^{+}(M') + d^{-}(M) d^{-}(M')$ defined by
\[
Z^{+} = (X_i^{+}\otimes Y_l^{+},~X_i^{-}\otimes Y_l^{-} : w(X_i^{\pm}) + w(Y_l^{\pm}) < k_{q^{+}}).
\]

Similarly, one observes that $c^{-}(M \otimes M')$ is the determinant of the square matrix $Z^-$ of size
 $d^{+}(M) d^{-}(M') + d^{-}(M) d^{+}(M')$ defined by
\[
Z^{-} = (X_i^{+} \otimes Y_l^{-}, ~X_i^{-}\otimes Y_l^{+}  : w(X_i^{\pm}) + w(Y_l^{\pm}) < k_{q^{-}}).
\]

The determinants of $Z^{\pm}$ can be expressed in the form $h^{\pm}(X,Y),$ where $h^{\pm}(x,y)$ are polynomial functions. For a fixed $y$, the function $h^{\pm}(x,y)$  has admissibility type
  \[
  \left\{(a^{\pm}_{\mu} : 1 \leq \mu \leq m, (d^{\pm}(M'), d^{\mp}(M'))\right\}.
  \]

Define the integers $(a^{\ast})^{\pm}_{\nu}$ for the motive $M'$ analogous to (\ref{amu}) above.
It follows that the data $(\left\{(a^{\ast})^{\pm}_{\nu} \right\}),(d^{\pm}(M), d^{\mp}(M))$ describes the admissibility type
 of $h^{\pm}(x,y)$ for a fixed $x$.

From the uniqueness property \cite[Theorem 1]{Yo}, it follows that the polynomials $h^{\pm}(x,y)$ can be expressed as $h^{\pm}(x,y) = \phi^{\pm}(x) \psi^{\pm}(y)$ where  $\phi^{\pm}(x)$ and $\psi^{\pm}(y)$ are polynomials with the following admissibility types respectively:
  \begin{equation} \label{adm1} \left\{a^{\pm}_{\mu} : ~(d^{\pm}(M'), d^{\mp}(M'))\right\}, \end{equation}
  \begin{equation} \label{adm2} \left\{(a^{\ast})^{\pm}_{\nu} : ~(d^{\pm}(M), d^{\mp}(M)) \right\}.\end{equation}

\medskip

\section{Period relations for motives over $\Q$}
\label{sec:main-q}

\subsection{}
We now state and prove one of the main results of this paper.

\begin{theorem}
\label{thm:main-q}
Let $M$ and $M'$ be pure motives over $\Q$ with coefficients in a number field $\E$. Suppose they satisfy the following properties:
\begin{enumerate}
\item rank($M$) = $d(M)$ is even, and rank($M'$) = $d(M')$ is odd.
\item All the nonzero Hodge numbers $h^{p,q}_{M}$ and $h^{p,q}_{M'}$ are equal to 1.
\item The tensor product motive $M \otimes M'$ is critical.
\end{enumerate}
Hypothesis (1) and (2) imply that $M$ and $M'$ are critical, and furthermore that complex conjugation acts as a scalar, denoted
$\epsilon(M')$, on the one-dimensional middle Hodge type of $M'$. Then the periods  $c^{\pm}(M)$ and $c^{\pm}(M \otimes M')$ are related by the following equation in
$(\E \otimes_\Q \C)^\times/\E^\times$:
\[ \frac{c^{+}(M\otimes M')}{c^{-}(M\otimes M')}  \ = \ \left( \frac{c^{+}(M)}{c^{-}(M)}\right)^{\epsilon(M')}. \]
\end{theorem}

\smallskip

\subsection{Proof of Theorem~\ref{thm:main-q}}
\label{sec:proof}

The condition on Hodge numbers of $M$ and $M'$ guarantees that the motives $M$ and $M'$ are critical, and
$d^{+}(M) = d^{-}(M) = d(M)/2$ and  $d^{+}(M') =   d^{-}(M') \pm 1$; indeed,
$d^{+}(M') = d^{-}(M') + \epsilon(M')$. Consider the motive $M\otimes M'$. Since it is critical and of even rank, it follows that $d^{\pm}(M \otimes M') = d(M)d(M')/2.$

\smallskip

Let $X$ and $Y$ be the period matrices of $M$ and $M',$ resp. The period $c^{\pm}(M\otimes M')$ is given by a polynomial $h^{\pm}(X,Y)$.
 Here $h^{\pm}(x,y) = \phi^{\pm}(x) \psi^{\pm}(y)$ and the polynomials $\phi^{\pm}(x)$ and $\psi^{\pm}(y)$ are of certain admissible types. The desired property follows from the analogous relation between the invariant polynomials $\phi^\pm(x)$, $\psi^\pm(y)$ and $f^\pm(x)$. To prove it, we compare their admissibility types under the hypothesis of the theorem.

\smallskip

Since $d^{+}(M) = d^{-}(M)$ and  $d^{+}(M') = d^{-}(M') + \epsilon(M')$,
we have  $\p_{M}^{+} = \p_{M}^{-} $, $\p_{M'}^{+} = \p_{M'}^{-} + \epsilon(M')$,  $d^{+}(M \otimes M') = d^{-}(M \otimes M')$ and $q^+ = q^{-}$. As a result, we have the following relations:
  \begin{equation}
  \label{mu}a^{+}_{\mu} = a^{-}_{\mu} \quad {\rm for} \ \ 1 \leq \mu \leq d(M),
  \end{equation}
\begin{equation}
\label{nu} (a^{\ast})^{+}_{\nu} = (a^{\ast})^{-}_{\nu} \quad {\rm for} \ \  1 \leq \nu \leq d(M') .
\end{equation}

\smallskip

 From (\ref{adm1}), (\ref{adm2}), (\ref{mu}) and (\ref{nu}), it follows that the admissibility types (and hence the functions themselves in view of their uniqueness) of $\psi^{+}(y)$ and $\psi^{-}(y)$ are equal up to $\Q^\times$-multiples, which we write as
 $$
 \psi^+(y) \approx \psi^-(y).
 $$
 Furthermore, the above conditions also imply that the admissibility types of $\phi^\pm(x)$ and $f^{\pm}(x)$ are related as
 we now explain. It is convenient to consider two cases: \\

\noindent \textbf{Case (i):}  $\epsilon(M') =1$.\\

 Here $d^{+}(M') = d^{-}(M') + 1$. From Lemma~\ref{multi}, and (\ref{adm1}), (\ref{adm2}), (\ref{mu}) and (\ref{nu}),
 we see that the admissibility types of $\phi^{+}(x)f^{-}(x)$ and $\phi^{-}(x)f^{+}(x)$ are given respectively by:
 \[ \left\{ (a^{+}_{1} + 1, ~ a^{+}_{2} + 1, \ldots a^{+}_{d(M)/2} + 1,~ a^{+}_{1+(d(M)/2)}, \ldots,
    a^{+}_{d(M)}),~(d^{+}(M'),~ d^{-}(M')+1) \right\},    \]
 \[ \left\{ (a^{+}_{1} + 1, ~ a^{+}_{2} + 1, \ldots a^{+}_{d(M)/2} + 1, ~a^{+}_{1+(d(M)/2)}, \ldots,
     a^{+}_{d(M)}),~(d^{-}(M')+1,~ d^{+}(M')) \right\}\]
which are identical.  Hence, from the uniqueness property of invariant polynomials of a given admissibility type we have
    \[ \phi^{+}(x) f^{-} (x) \ \approx \ \phi^{-}(x) f^{+} (x), \ \ (\text{equality up to}~ \Q^{\times}).\]
Hence we get
    \[
    \phi^{+}(x)  \psi^+(y) f^{-} (x) \ \approx \ \phi^{-}(x)  \psi^-(y)f^{+} (x),
    \]
an equality of polynomials up to $\Q^\times$; evaluating on period matrices, we get
\[
c^+(M \otimes M') c^-(M) \ \approx \ c^-(M \otimes M') c^+(M),
\]
which is an equality in $R^\times = (\E \otimes_\Q \C)^\times$ up to $\E^\times$. This concludes the proof in case (i). \\

\noindent \textbf{Case (ii):}  $\epsilon(M') = -1$.\\

 Here $d^{+}(M') = d^{-}(M') - 1$. From an analogous argument as in the previous case we get
    \[ \phi^{+}(x) f^{+} (x) \ \approx \ \phi^{-}(x) f^{-} (x).\]
The rest is similar, and in this case we end up with
 \[
 c^+(M \otimes M') c^+(M) \ \approx \ c^-(M \otimes M') c^-(M).
\]

\medskip

\subsection{} \label{sec:hodge-one}
 Our proof of Theorem~\ref{thm:main-q} relies on the facts that $d^{+}(M) = d^{-}(M)$ and $d^{+}(M') = d^{-}(M') \pm 1$. The assumption on the Hodge numbers in Theorem~\ref{thm:main-q} guarantees these conditions which in turn are valid for the motives coming from cohomological cuspidal representations. Indeed,
Theorem~\ref{thm:main-q} can be rephrased by the foregoing conditions on the dimensions $d^{\pm}(M)$ and
$d^{\pm}(M')$ in the hypotheses.

We may further relax the hypotheses on $M'$. Suppose $M'$ is a pure critical motive of odd rank.
Then, complex conjugation acts by a scalar $\epsilon(M')$ on the middle Hodge type, and
$d^{+}(M') = d^{-}(M') + \epsilon(M')k$, where $k$ is the dimension of this middle Hodge type of $M'$. Then the same proof gives:
\begin{equation}
\label{eqn:variation}
\frac{c^{+}(M\otimes M')}{c^{-}(M\otimes M')}  \ = \ \left( \frac{c^{+}(M)}{c^{-}(M)}\right)^{\epsilon(M')k}.
\end{equation}

\medskip

\section{Period relations for motives over totally real fields}
\label{sec:motives-tot-real}

\subsection{A factorization result for Deligne's periods over a totally real field}
Let $\F$ be a totally real number field. Let $I_{\F}$ be the set of all real embeddings of $\F$ into $\C$. A motive $M$ over $\F$ with coefficients in a number field $\E$ has the following realizations: For each $\sigma \in I_{\F}$ we have a
Betti realization of $M$, denoted $H_{B}(\sigma,M),$  which is a vector space of dimension $d(M)$ over $\E$ together with an action of the complex conjugation which we denote $\rho_\sigma.$
The  de~Rham realization of $M$, denoted $H_{DR}(M)$, is a free $\E \otimes_\Q \F$ module of rank $d(M)$ with a decreasing filtration $F^p_{DR}(M)$ of $\E \otimes_\Q \F$-modules.
For each $\sigma \in I_{\F}$, there is a comparison isomorphism of $\E \otimes_\Q \C$-modules
\[
I_{\sigma} : H_{B}(\sigma, M) \otimes_{\Q} \C \ \longrightarrow \ H_{DR}(M) \otimes_{\F,\sigma} \C.
\]
There is a Hodge decomposition: $H_{B}(\sigma, M) \otimes_{\Q} \C = \oplus_{p,q} H^{p,q}_\sigma(M)$ and
$\rho_{\sigma}$ maps $H^{p,q}_\sigma(M)$ onto $H^{q,p}_\sigma(M)$. The rank of $ H^{p,q}_\sigma(M)$ is independent of $\sigma$ and is denoted by $h^{p,q}_M$; these are the Hodge numbers of $M.$

Suppose $M$ is critical, then for each $\sigma \in I_\F$ the periods $c^{\pm}(\sigma, M)$ are defined in a manner analogous to the case of motives over $\Q$, and for a given $\sigma$, the periods $c^{\pm}(\sigma, M)$ are well-defined as elements of $(\E \otimes \C)^{\times}$ mod $(\E\otimes \sigma(\F))^{\times}$.

\smallskip

Given a motive $M$ over $\F$ with coefficients in $\E$, the restriction of scalars functor gives a motive $R_{\F|\Q}(M)$ over $\Q$ with coefficients in $\E$ such that
\begin{enumerate}
\item $H_{DR}(R_{\F|\Q}(M)) = H_{DR}(M)$ as an $\E$-vector space of dimension $d(M) [\F : \Q]$, and
\item $H_{B}(R_{\F|\Q}(M)) = \bigoplus \limits_{\sigma \in I_{\F}} H_{B}(\sigma, M).$
\end{enumerate}
The periods $ c^{\pm}(R_{\F|\Q}(M))$ have the following factorization:
\begin{equation}
\label{factor}
 c^{\pm}(R_{\F|\Q}(M)) = (1 \otimes D^{1/2}_{\F})^{d^{\pm}(M)} \prod \limits_{\sigma \in I_{\F}} c^{\pm}(\sigma, M)
 \quad  \pmod{\E^{\times}}.
 \end{equation}
Here $D_{\F}$ is the absolute discriminant of $\F$. Such a factorization of periods has been observed by many;
see, for example, Blasius~\cite[M.8]{blasius}, Hida~\cite[p.442]{Hi} or
Panchishkin~\cite[p.995]{Pa}. This is closely related to a long history concerning Shimura's conjecture about factorization of periods related to  Hilbert modular forms; we refer the reader to Yoshida ~\cite{YoExpo} and Harris~\cite{harris} and the references therein.

\smallskip

\subsection{Period relations over totally real fields}
An analogue of Theorem \ref{thm:main-q} holds  for  motives over totally real fields under suitably modified hypotheses.

\begin{theorem}
\label{thm:totreal}
Let $M$ and $M'$ be motives defined over a totally real number field $\F$ with coefficients in a number field $\E$.
Suppose they satisfy the following properties:
\begin{enumerate}
\item rank($M$) = $d(M)$ is even, and rank($M'$) = $d(M')$ is odd.
\item All the nonzero Hodge numbers $h^{p,q}_{M}$ and $h^{p,q}_{M'}$ are equal to 1.
\item The motive $M \otimes M'$ is critical.
\end{enumerate}
Let $\epsilon(\sigma,M')$ be the scalar by which the complex conjugation $\rho_{\sigma}$ acts on the rank-one middle Hodge type of $H_B(\sigma, M')$.
\begin{enumerate}
\item[(i)]
For $\sigma \in I_{\F}$, as elements of $(\E \otimes \C)^{\times}$  we have
\[ \frac{c^{+}(\sigma, M\otimes M')}{c^{-}(\sigma,M\otimes M')} \ = \
\left( \frac{c^{+}(\sigma,M)}{c^{-}(\sigma,M)}\right)^{\epsilon(\sigma, M')}
\pmod{(\E\otimes \sigma(\F))^{\times}}.
\]
\item[(ii)] As elements of $(\E \otimes \C)^{\times}$  we have
\[
\frac{c^{+}(R_{\F|\Q}(M \otimes M'))}{ c^{-}(R_{\F|\Q}(M \otimes M'))} \ = \
\prod \limits_{\sigma \in  I_{\F}} \left( \frac{ c^{+}(\sigma, M)} {c^{-}(\sigma, M)} \right) ^{\epsilon(\sigma,M')}
\pmod{(\E\otimes \mathcal{F})^{\times}},
\]
where $\mathcal{F}$ is any subfield of $\C$ containing $\sigma(\F)$ for all $\sigma \in I_\F.$
\end{enumerate}
\end{theorem}

\begin{proof}
The proof of Theorem \ref{thm:main-q} gives the proof of (i) {\it mutatis mutandis} since the discussion involving Yoshida's results (in \S\ref{sec:yoshida-polynomials}) works over $\Q$. Next, (i) $\Longrightarrow$ (ii) follows from (\ref{factor}); note that the discriminant factor cancels out since
$d^+(M\otimes M') = d^-(M\otimes M') = d(M)d(M')/2.$
\end{proof}

\begin{remark}{\rm
Panchishkin has conjectured, based on suggestions from Beilinson, that there should exist
$\tilde{c}^{\pm}(\sigma, M) \in (\E \otimes \C)^\times$ well-defined modulo $\E^\times$ such that
for all $\sigma \in I_{\F}$ we have
\[
\tilde{c}^{\pm}(\sigma, M) \ = \  c^{\pm}(\sigma, M) \quad \pmod{(\E \otimes \sigma(F))^{\times}}.
\]
(See \cite[Conjecture 2.3]{Pa}.) Granting this, statement (ii) of Theorem~\ref{thm:totreal} can be conjecturally refined
as:
\[
\frac{c^{+}(R_{\F|\Q}(M \otimes M'))}{ c^{-}(R_{\F|\Q}(M \otimes M'))} \ = \
\prod \limits_{\sigma \in  I_{\F}} \left( \frac{ \tilde{c}^{+}(\sigma, M)} {\tilde{c}^{-}(\sigma, M)} \right) ^{\epsilon(\sigma,M')}
\pmod{\E^{\times}}.
\]
}\end{remark}

\section{Comments on when ranks of $M$ and $M'$ have the same parity}
\label{sec:parity}

\subsection{When the ranks of both $M$ and $M'$ are even.}

In this case, it follows from the hypotheses of Theorem~\ref{thm:main-q} that $d^{+}(M) = d^{-}(M)$, $d^{+}(M') = d^{-}(M')$ and $d^{+}(M \otimes M') = d^{-}(M \otimes M')$. Since the tensor product motive $M \otimes M'$ is critical by assumption, there exist integers $q^{+},~ q^{-}$ (as defined in
(\ref{hodgetensor})) such that
\[
u_1 + u_2 + \ldots + u_{q^{\pm}} \ = \ d^{\pm}(M \otimes M').
\]
Thus we have $q^{+} = q^{-}$ and as a result, equations (\ref{mu}) and (\ref{nu}) are satisfied. From the results of Yoshida (\textit{cf.}~\cite[Cor.1, p.1188]{Yo}) we get
\begin{equation}
\label{eqn:even-ranks}
c^{+}(M \otimes M') \ = \ c^{-}(M \otimes M').
\end{equation}

This ties up very well with known results on critical values. Consider a classical situation: suppose ${\sf f}$ and
${\sf g}$ are primitive holomorphic cusp forms of the same level and of weights $k$ and $l$. Suppose $k < l$ and $\psi$ is the nebentypus of ${\sf g}$. Let us look at the (degree four) Rankin-Selberg $L$-function
$L(s, {\sf f} \times {\sf g})$. Let $l \leq m < k.$ Then Shimura \cite[Theorem 4]{shimura} has proved
$$
L_f(m, {\sf f} \times {\sf g}) \ \sim \ (2 \pi i)^{2m-l-1} \gamma(\psi) u^+({\sf f})u^-({\sf f}).
$$
Suppose now that both $m$ and $m+1$ are critical then we get
$$
\frac{L_f(m, {\sf f} \times {\sf g})} {L_f(m+1, {\sf f} \times {\sf g})} \ \sim \ (2 \pi i)^{-2}.
$$
As in \S\ref{sec:classical} we can absorb the $(2 \pi i)^2$ into the ratio of $L$-factors at infinity and deduce:
\begin{equation}
\label{eqn:gl2-gl2}
L(m, {\sf f} \times {\sf g}) \ \sim \ L(m+1, {\sf f} \times {\sf g}).
\end{equation}
It is this statement about $L$-values for $\GL_2 \times \GL_2$ that is motivically interpreted in
(\ref{eqn:even-ranks}) for a tensor product of two rank two motives. The generalization of (\ref{eqn:gl2-gl2}) to the context of Rankin-Selberg $L$-functions for $\GL_n \times \GL_{n'}$, when both $n$ and $n'$ are even, is work in progress by
G\"unter~Harder and the second author; the results will appear elsewhere.

\subsection{}
Theorem \ref{thm:main-q}, (\ref{eqn:variation}) and (\ref{eqn:even-ranks}) can be combined together as:
\begin{theorem}
\label{thm:combined-even-odd}
Let $M$ and $M'$ be pure motives over $\Q$ with coefficients in a number field $\E$. Suppose they satisfy the following properties:
\begin{enumerate}
\item $M$ is critical and $d^{+}(M) = d^{-}(M)$.
\item $M'$ is a direct sum of critical motives.
\item $M \otimes M'$ is critical. 
\end{enumerate}
Then the periods  $c^{\pm}(M)$ and $c^{\pm}(M \otimes M')$ are related by the following equation in
$(\E \otimes_\Q \C)^\times/\E^\times$:
\[ \frac{c^{+}(M\otimes M')}{c^{-}(M\otimes M')}  \ = \ \left( \frac{c^{+}(M)}{c^{-}(M)}\right)^{\text{Tr}(\rho|_{H_{B}(M')})}. \]
\end{theorem}
We leave it to the reader to formulate the analogous statement for motives over a totally real field $\F$ with coefficients in $\E.$

\medskip

\subsection{When the ranks of both $M$ and $M'$ are odd.}

Let $M,M'$ be pure motives defined over $\Q$ with coefficients in a number field $\E$.  Suppose that both the ranks  $d(M)$ and $d(M')$ are odd. Similar to the earlier even-odd case, assume that all Hodge numbers $h^{p,q}(M),h^{p,q}(M')$ are less than or equal to one. From this it follows that $M$ and $M'$ are critical and they have non-zero middle Hodge types. Let $\epsilon(M)$ (resp., $\epsilon(M')$) be the scalar by which complex conjugation acts on the middle Hodge type of $M$ (resp., $M'$). Suppose the tensor product motive $M \otimes M'$ is also critical. It follows that
\begin{equation} \label{dim} d^{+}(M \otimes M') - d^{-}(M \otimes M') = \epsilon(M) \epsilon(M').
\end{equation}
Consider the Hodge filtrations on the de~Rham realizations of the motives $M, M'$ and $M\otimes M'$.
  Let $u_t$ be the $t^{th}$ Hodge number of $M \otimes M'$, i.e., the $\E$-dimension of $F^{k_{t}}(M \otimes M')/F^{k_{t+1}}(M \otimes M')$.
Let $1 \leq r \leq d(M)$, $1 \leq s \leq d(M')$. Let $\left\{ v_r\right\}$, $\left\{w_s\right\}$ be $\E$-basis of the one dimensional quotient spaces $F^{i_{r}}(M)/F^{i_{r+1}}(M)$, $F^{j_{s}}(M')/F^{j_{s+1}}(M')$ respectively. Then it follows that the set  $\left\{v_r \otimes w_s : i_r + j_s = k_t \right\}$ of classes modulo $F^{k_{t+1}}(M \otimes M')$ forms a basis of the quotient space $F^{k_{t}}(M \otimes M')/F^{k_{t+1}}(M \otimes M')$. The size of this set is $u_t$.
Since the motive $M \otimes M'$ is critical, there exist integers $q^{\pm}$ such that 
$$d^{\pm}(M \otimes M) = \sum \limits_{t \leq q^{\pm}} u_t.$$
Thus,  $d^{\pm}(M \otimes M) = |\left\{ (r,s) : 1 \leq r \leq d(M), 1 \leq s \leq d(M'), i_{r}+j_{s} \leq k_{q^{\pm}} \right\}|$.
From (\ref{dim}) above, we have $q^{+} = q^{-} + \epsilon(M) \epsilon(M'), ~u_{\text{max}(q^{+},q^{-})} =1$. For $1 \leq r \leq d(M)$, let $a^{\pm}_r = |\left\{ s : 1 \leq s \leq d(M'),~ i_r + j_s \leq k_q^{\pm} \right\}|$.
\vspace{0.5 cm}

\noindent \textbf{Case 1} : $q^+ = q^{-} + 1$.\\

\noindent We have
 $$ a^{+}_r  - a^{-}_r = |\left\{ s : 1 \leq s \leq d(M'),~  k_q^{-} < i_r + j_s \leq k_q^{+} \right\}|.$$
 Since $u_{q^{+}} =1$, it follows that there exists unique $(r_0, s_0)$ such that
 \begin{itemize}
 \item $1 \leq r_0 \leq d(M),~ 1 \leq s_0 \leq d(M')$.
 \item $i_{r_0} + j_{s_0} = k_{q^{+}}$.
 \item $ a^{+}_{r_0}  - a^{-}_{r_0} = 1$.
 \item $ a^{+}_r  = a^{-}_r, \quad \forall ~ r \neq r_0.$
 \end{itemize}
(In fact $r_0 = d^{-}(M)+ (\epsilon(M) +1)/2$ and $s_0 = d^{-}(M')+ (\epsilon(M') +1)/2$.) \\

\noindent \textbf{Case 2} : $q^+ = q^{-} - 1$.\\

\noindent We have
 $$ 
 a^{-}_r  - a^{+}_r = |\left\{ s : 1 \leq s \leq d(M'),~  k_q^{+} < i_r + j_s \leq k_q^{-} \right\}| .
 $$
Since $u_{q^{-}} =1$, it follows that there exists unique $(r_0, s_0)$ (same as in Case 1) such that
 \begin{itemize}
 \item $1 \leq r_0 \leq d(M),~ 1 \leq s_0 \leq d(M')$.
 \item $i_{r_0} + j_{s_0} = k_{q^{-}}$.
 \item $ a^{-}_{r_0}  - a^{+}_{r_0} = 1$.
 \item $ a^{+}_r  = a^{-}_r \quad \forall ~ r \neq r_0.$ \\
 \end{itemize}

Define the ``dual'' data as follows. For $ 1 \leq s \leq d(M')$, let $a^{\ast,\pm}_s = |\left\{ 1 \leq r \leq d(M) : a^{\pm}_r \geq s \right\}|$. It follows (from
 an argument similar to the case of $a^{\pm}_r$) that, 
 $a^{\ast, +}_{s_{0}}  - a^{\ast, -}_{s_{0}}  = q^{+} - q^{-}$  and $a^{+}_{s} = a^{-}_s \quad \forall ~ s \neq s_0.$

 The following is the analogue of Theorem \ref{thm:main-q} in the situation where both motives are of odd rank.

 \begin{theorem}
\label{thm:main-odd}
Let $M$ and $M'$ be pure motives over $\Q$ with coefficients in a number field $\E$. Suppose they satisfy the following properties:
\begin{enumerate}
\item rank($M$) = $d(M)$ and rank($M'$) = $d(M')$ are odd.
\item All the nonzero Hodge numbers $h^{p,q}_{M}$ and $h^{p,q}_{M'}$ are equal to 1.
\item The tensor product motive $M \otimes M'$ is critical.
\end{enumerate}
Hypothesis (2) implies that the complex conjugation acts as a scalar, denoted
$\epsilon(M)$, (resp., $\epsilon(M')$), on the one-dimensional middle Hodge type of $M$ (resp., $M'$).
\smallskip
Then the periods  $c^{\pm}(M)$, $c^{\pm}(M')$  and $c^{\pm}(M \otimes M')$ are related by the following equation in
$(\E \otimes_\Q \C)^\times/\E^\times$:
\[ \frac{c^{+}(M\otimes M')}{c^{-}(M\otimes M')}  \ = \ \left( \frac{c^{+}(M)}{c^{-}(M)}\right)^{\epsilon(M')} \left( \frac{c^{+}(M')}{c^{-}(M')}\right)^{\epsilon(M)}. \]
\end{theorem}

\medskip

\begin{proof}
From the definition of the Deligne periods, it follows that $c^{+}(M \otimes M')$ is the determinant of the square matrix $Z^+$ of
 size $d^{+}(M) d^{+}(M') + d^{-}(M) d^{-}(M')$ defined by
\[
Z^{+} = (X_r^{+}\otimes Y_s^{+},~X_r^{-}\otimes Y_s^{-} : 1 \leq r \leq d(M), ~ 1 \leq s \leq d(M') : i_r + j_s \leq k_{q^{+}}).
\]
Similarly, one observes that $c^{-}(M \otimes M')$ is the determinant of the square matrix $Z^-$ of size
 $d^{+}(M) d^{-}(M') + d^{-}(M) d^{+}(M')$ defined by
\[
Z^{-} = (X_r^{+} \otimes Y_s^{-}, ~X_r^{-}\otimes Y_s^{+}  : 1 \leq r \leq d(M), ~ 1 \leq s \leq d(M') : i_r + j_s \leq k_{q^{-}}).
\]

Recall that the determinants of $Z^{\pm}$ can be expressed in the form $h^{\pm}(X,Y),$ where $h^{\pm}(x,y)$ are polynomial functions. For a fixed $y$, the function $h^{\pm}(x,y)$  has admissibility type
  \[
  \left\{(a^{\pm}_{\mu} : 1 \leq \mu \leq d(M)), (d^{\pm}(M'), d^{\mp}(M'))\right\}.\]
%
The polynomials $h^{\pm}(x,y)$ can be expressed as $h^{\pm}(x,y) = \phi^{\pm}(x) \psi^{\pm}(y)$ where  $\phi^{\pm}(x)$ and $\psi^{\pm}(y)$ are polynomials with the following admissibility types respectively (\ref{adm1}, \ref{adm2}):
  \begin{equation}  \left\{a^{\pm}_{\mu} : 1 \leq \mu \leq d(M)), ~(d^{\pm}(M'), d^{\mp}(M'))\right\}, \end{equation}
 \begin{equation} \left\{(a^{\ast, \pm}_{\nu} : 1 \leq \nu \leq d(M')),~(d^{\pm}(M), d^{\mp}(M)) \right\}.\end{equation}

The proof follows from the same arguments those in the proof of Theorem \ref{thm:main-q}. We give the proof for the case when $\epsilon(M) =\epsilon(M') =+1$ as a sample case. 

Here $q^+ = q^- + 1$. So $a^{+}_r = a^{-}_r$ except when $r =d^{+}(M)$ and $a^{+}_{d^{+}(M)} = a^{-}_{d^{+}(M)} + 1$. Similarly $a^{\ast,+}_s = a^{\ast, -}_s$ except when $s =d^{+}(M')$ and $a^{\ast,+}_{d^{+}(M')} = a^{\ast,-}_{d^{+}(M')} + 1$. It can be seen that both $\phi^{+} f^{-}$ and $\phi^{-} f^{+}$ are of same admissibility type given by
{\small
$$
\left\{(a^{+}_{1} + 1, ~a^{+}_{2}+1, ~\ldots, ~a^{+}_{d^{-}(M)}+1,~ a^{+}_{d^{+}(M)} = a^{-}_{d^{+}(M)} + 1, ~\ldots,~ a^{+}_{d(M)}),(d^{+}(M'),~d^{+}(M'))\right\}.
$$}
On the other hand, $\psi^{+} g^{-}$ and $\psi^{-} g^{+}$ are of same admissibility type given by 
{\small
 $$
 \left\{(a^{\ast,+}_{1} + 1, ~a^{\ast, +}_{2}+1, ~\ldots, ~a^{\ast,+}_{d^{-}(M')}+1,~ a^{\ast, +}_{d^{+}(M')} = a^{\ast, -}_{d^{+}(M')}+1, ~\ldots,~ a^{\ast, +}_{d(M')}),(d^{+}(M),~d^{+}(M))\right\}.
 $$}
 The desired relation follows from the uniqueness results in \cite{Yo}.
 \end{proof}

It is an amusing exercise to use the special values of Dirichlet $L$-functions 
(see, for example, Neukirch \cite[p.442]{Neu}) and the special values of the symmetric square $L$-functions attached to holomorphic primitive elliptic cusp forms (due to Sturm \cite{sturm}) to produce examples illustrating 
Theorem~\ref{thm:main-odd} when ${\rm rank}(M)=1$ and ${\rm rank}(M') = 1 \ {\rm or} \ 3.$


\begin{thebibliography}{1000000}

\bibitem[Bla87]{blasius-orloff} Blasius, D.,  
{\it Appendix to Orloff Critical values of certain tensor product $L$-functions.}
Invent. Math., 90, (1987), no.1, 181--188.

\bibitem[Bla97]{blasius}
Blasius, D., 
{\it Period relations and
critical values of $L$-functions.}
Olga Taussky-Todd: in memoriam.
Pacific J. Math., (1997), Special Issue, 53--83.


\bibitem[Clo90]{clozel}
Clozel, L., 
{\it Motifs et formes automorphes: applications du principe de
fonctorialit\'e.} (French) [Motives and automorphic forms: applications of
the functoriality principle]
Automorphic forms, Shimura varieties, and $L$-functions, Vol. I (Ann Arbor, MI, 1988),
77--159, Perspect. Math., 10, Academic Press, Boston, MA, 1990.


\bibitem[Del79]{deligne}
Deligne, P., 
{\it Valeurs de fonctions $L$ et p\'eriodes d'int\'egrales (French),}
With an appendix by N. Koblitz and A. Ogus.
Proc. Sympos. Pure Math., XXXIII, Automorphic forms, representations and $L$-functions (Proc.
Sympos. Pure Math., Oregon State Univ., Corvallis, Oregon, 1977), Part 2,
pp. 313--346, Amer. Math. Soc., Providence, R.I., 1979.

\bibitem[Hard10]{harder}
Harder, G., {\it Arithmetic aspects of rank one Eisenstein cohomology.} Cycles, motives and Shimura varieties, 131--190, Tata Inst. Fund. Res. Stud. Math., Tata Inst. Fund. Res., Mumbai, 2010.

\bibitem[HaRa11]{harder-raghuram}
Harder, G., and Raghuram, A.,
{\it Eisenstein cohomology and ratios of critical values of Rankin-Selberg $L$-functions.}
C. R. Acad. Sci. Paris, Ser. I 349 (2011) 719--724.

\bibitem[Harr93]{harris}
Harris, M.,
{\it L-functions of  $2 \times 2$ unitary groups and factorization of periods of Hilbert modular forms.}
Jour. of AMS, Volume 6, Number 3 (1993) 637--719.

\bibitem[Harr97]{harris97}
Harris, M.,
{\it $L$-functions and periods of polarized regular motives.}
J. Reine Angew. Math., 483, (1997), 75–-161.

\bibitem [Hid94]{Hi} Hida, H., \emph{On the critical values of $L$-functions of $GL(2)$ and 
$\text{GL}(2) \times \text{GL}(2)$}, Duke Mathematical Journal, 74(2) (1994), 431--529.

\bibitem [Neu99]{Neu} Neukirch, J., {\it Algebraic number theory}, Springer-Verlag Berlin Heidelberg, (1999).

\bibitem [Pan94] {Pa} Panchishkin, A., \emph{Motives over totally real fields and $p$-adic $L$-functions}, Annales de l'Institut Fourier, 44(4) (1994), 989-1023.


\bibitem[Shi76]{shimura76}
G. Shimura,
{\it The special values of zeta functions associated with cusp forms,}
Comm. Pure and Appl. math., vol. XXIX, 783--804, (1976).

\bibitem[Shi77]{shimura}
G. Shimura,
{\it On the periods of modular forms,}
Math. Ann.,  229, no. 3, 211--221, (1977).

\bibitem[Stu80]{sturm}
 Sturm, J., \emph{Special values of zeta functions, and Eisenstein series of half integral weight.}
 Amer. J. Math. 102 (1980), no. 2, 219--240.

\bibitem [Yos95] {YoExpo}
Yoshida, H.,
\emph{On a conjecture of Shimura concerning periods of Hilbert modular forms}, Amer.\ J.\ Math., 117 (1995), no.4, 1019--1038. 

\bibitem [Yos01] {Yo} Yoshida, H., \emph{Motives and Siegel modular forms}, Amer.\ J.\ Math., 123 (2001), 1171--1197.



\end{thebibliography}
\end{document}